\documentclass[12pt]{amsart}
\usepackage{amssymb}
\usepackage{amsfonts}
\usepackage{amssymb,latexsym}
\usepackage{enumerate}
\usepackage{mathrsfs}
\usepackage{hyperref}
 \usepackage{color}\makeatletter
\@namedef{subjclassname@2010}{%
  \textup{2010} Mathematics Subject Classification}
\makeatother

  [2002/01/19 v2.2g %
    AMS font definitions%
  ]
\DeclareFontFamily{U}{euf}{}
\DeclareFontShape{U}{euf}{m}{n}{%
  <5><6><7><8><9>gen*eufm%
  <10><10.95><12><14.4><17.28><20.74><24.88>eufm10%
  }{}
\DeclareFontShape{U}{euf}{b}{n}{%
  <5><6><7><8><9>gen*eufb%
  <10><10.95><12><14.4><17.28><20.74><24.88>eufb10%
  }{}

  [2002/01/19 v2.2g %
    AMS font definitions%
  ]
\DeclareFontFamily{U}{msb}{}
\DeclareFontShape{U}{msb}{m}{n}{%
  <5><6><7><8><9>gen*msbm%
  <10><10.95><12><14.4><17.28><20.74><24.88>msbm10%
  }{}

  [2002/01/19 v2.2g %
    AMS font definitions%
  ]
\DeclareFontFamily{U}{msa}{}
\DeclareFontShape{U}{msa}{m}{n}{%
  <5><6><7><8><9>gen*msam%
  <10><10.95><12><14.4><17.28><20.74><24.88>msam10%
  }{}

\newtheorem{theorem}{Theorem}[section]
\newtheorem{lemma}[theorem]{Lemma}

\newtheorem{corollary}[theorem]{Corollary}

\theoremstyle{definition}

\newtheorem{remark}[theorem]{Remark}

\numberwithin{equation}{section} 

\begin{document}

\title[]
{Asymptotics of generalized partial theta functions with a Dirichlet character}



\author{Su Hu}
\address{Department of Mathematics, South China University of Technology, Guangzhou, Guangdong 510640, China}
\email{mahusu@scut.edu.cn}

\author{Min-Soo Kim}
\address{Department of Mathematics Education, Kyungnam University, Changwon, Gyeongnam 51767, Republic of Korea}
\email{mskim@kyungnam.ac.kr}




\begin{abstract}
In this paper, we prove asymptotic expansions of generalized  partial theta
functions with a nonprincipal Dirichlet character and relate these expansions to certain $L$-series. \end{abstract}

\maketitle

\section{Introduction}
\subsection{Background} 
The Jacobi theta function is defined by
$$\theta_{3}(u;q):=1+2\sum_{n=1}^{\infty} q^{n^2}\cos(2nu)=\sum_{n=-\infty}^{\infty}x^{n}q^{n^2},$$
where $u,q\in\mathbb{C}, 0<|q|<1$ and $x=e^{2\pi iu}$.
In the lost notebook, Ramanujan \cite{Ramanujan} presented several identities for functions  which are similar with
the Jacobi theta function. These functions have the form 
$$\sum_{n=0}^{\infty}x^{n}q^{n^2}$$ or $$\prod_{k=1}^{n}(1+xq^{k})(1+x^{-1}q^{k}),$$
which are named partial theta functions by Andrews in \cite{Andrews}.

Due to many predecessors'
efforts, the theory of partial theta functions has been deeply connected with many interesting topics in modern mathematics, such as the basic hypergeometric function theory \cite{Andrews,Andrews2}, the Rogers-Ramanujan continued fraction \cite{Andrews3},
the $q$-hypergeometric series \cite{BFR}, half-integral weight Eichler integrals and quantum modular forms \cite{BR}, Jacobi forms and the regularized characters of singlet algebra modules in the representation theory of vertex algebras \cite{BFM}, mock modular forms and quantum modular forms \cite{FOR, Kimport}, Hardy-Petrovitch-Hutchinson's problem in classical function theory \cite{KS}, quantum invariants of $3$-manifolds \cite{LZ}, Appell-Lerch sums \cite{Mortenson}, the generalized triple product identity \cite{Warnaar1}, higher rank partial theta functions and the representation theory \cite{CM}. For a 
complete history on partial theta functions and  their development, we refer to a recent survey article by Ole Warnaar~\cite{Warnaar}.

Let  $\theta>0,b$ real. In 2011, Berndt and Kim \cite{BK} introduced the following partial theta functions:
\begin{equation}G_1(\theta)=2\sum_{n=0}^\infty(-1)^ne^{-(n+b/2)^2\theta}\end{equation}
and   \begin{equation} G_2(\theta)=2\sum_{n=0}^\infty e^{-(n+b/2)^2\theta}. \end{equation}
They also obtained asymptotic expansions for $G_{1}(\theta)$ and $G_{2}(\theta)$ involving Bernoulli and Hermite polynomials (see \cite[Theorems 1.1 and 3.4]{BK}).
Their works were inspired by the following asymptotic expansion for the partial theta function in Ramanujan's second notebook \cite[p. 324]{Ramanujan}
\begin{equation}
\begin{aligned}
2\sum_{n=0}^{\infty}(-1)^{n}q^{n^{2}+n}&=2\sum_{n=0}^{\infty}(-1)^{n}\Big(\frac{1-t}{1+t}\Big)^{n^{2}+n} \\
&\sim 1+t+t^{2}+2t^{3}+5t^{4}+\cdots,
\end{aligned}
\end{equation}
where $q=\frac{1-t}{1+t}\to 1^{-}$, or $t\to 0^{+}$. As remarked by Berndt and Kim~\cite{BK}, the above asymptotic expansion is very interesting since there is no a priori reason to believe that the coefficients (in the variable $t$) are positive integers.
This phenomenon has now been well understood, see the work of Galway \cite{Galway}, of Stanley \cite{Stanley}, of Bringmann--Folsom \cite{BF}, etc.

During the recent years, there are also several other works concerning the asymptotic expansions of partial theta functions, see \cite{BF, BR, BFM, FOR, KS, Kimport, KKS, Mao}.

Let  $\theta>0,c$ real, and $r\in\mathbb N$. In 2018, McIntosh \cite{Mc} proved asymptotic expansions for  general sums
$$\sum_{n=0}^\infty e^{-(n+c)^r\theta}\quad\text{and}\quad\sum_{n=0}^\infty(-1)^n e^{-(n+c)^r\theta}.$$

Let  $\theta>0,b$ real, and $r\in\mathbb N$. In this paper, we shall  consider the following generalized partial theta functions with a Dirichlet character $\chi$:
\begin{equation}\label{sums-theta-1}
G_{1,\chi}(\theta,b,r)=r\sum_{n=0}^\infty(-1)^n\chi(n)e^{-(n+b/r)^r\theta}
\end{equation}
and
\begin{equation}\label{sums-theta-2}
G_{2,\chi}(\theta,b,r)=r\sum_{n=0}^\infty\chi(n)e^{-(n+b/r)^r\theta}.
\end{equation}
In the case $r=2$ and $\chi=\chi^0$ (the principal character, i.e., its conductor $f_\chi=1$ and $\chi(m)=1$ for all integers $m$, see \cite[p. 4]{Iwasawa}), (\ref{sums-theta-1}) and (\ref{sums-theta-2})
are Berndt and Kim's partial theta functions $G_1(\theta)$ and $G_2(\theta)$, respectively.  

By applying the character analogue of the Euler-Maclaurin summation formula, we obtain asymptotic expansions of (\ref{sums-theta-1}) and (\ref{sums-theta-2}) in terms of the generalized Bernoulli and Euler polynomials (see Theorems \ref{thm-partial theta fts} and \ref{thm2} below). Inspired by the works of Lawrence and Zagier~\cite{LZ}, we also show a connection between certain $L$-series and our asymptotic expansions (see Theorem~\ref{res-L-asy} below).
\subsection{The generalized Bernoulli and Euler polynomials} In this subsection, we shall review the definitions and some basic properties on the generalized Bernoulli and Euler polynomials, which play fundamental roles in the derivations of our asymptotic expansions.
 
The Bernoulli polynomials $B_n(x)$ are defined by the generating function
\begin{equation}\label{def-B-pol}
\frac{te^{tx}}{e^t-1}=\sum_{n=0}^\infty B_n(x)\frac{t^n}{n!}\quad(|t|<2\pi)
\end{equation}
and $B_n=B_n(0)$ are the Bernoulli numbers with $B_0=1,B_1=-1/2$ and $B_{2n+1}=B_{2n-1}(1/2)=0$ for $n\geq1.$

Suppose that $\chi$ is a primitive character modulo $f_\chi$ and $\bar\chi(n): =\overline{\chi(n)}$ is its complex conjugate.
The generalized Bernoulli polynomials $B_{n,\chi}(x)$
are defined by the following generating function (\cite[Proposition 6.2]{Ber75})
\begin{equation}\label{def-g-B-pol}
\sum_{m=0}^{f_\chi-1}\frac{\bar\chi(m)te^{(m+x)t}}{e^{f_\chi t}-1}=\sum_{n=0}^\infty B_{n,\chi}(x)\frac{t^n}{n!}\quad(|t|<2\pi/f_\chi)
\end{equation}
and $B_{n,\chi}=B_{n,\chi}(0)$ are the generalized Bernoulli numbers.
In particular, if $\chi^0$ is the principal character, i.e., if $f_\chi=1$ and $\chi(m)=1$ for all integers $m$
then $B_{n,\chi^0}(x)=B_{n}(x)$ for $n\geq0$ and $B_{0,\chi}(x)=0$ for $\chi\neq\chi^0.$

The generalized Bernoulli functions $\overline B_{n,\chi}(x)$ are functions with period $f_\chi.$  They are defined by
(\cite[Theorem 3.1]{Ber75})
\begin{equation}\label{def-B-fts}
\overline B_{n,\chi}(x)=f_\chi^{n-1}\sum_{m=0}^{f_\chi-1}\bar\chi(m)\overline B_n\left(\frac{m+x}{f_\chi}\right), \quad n\geq1
\end{equation}
for all real $x.$
Here we recall the following properties for the generalized Bernoulli functions which will be needed in the sequel (see \cite[(9), (10) and (11)]{DC})
\begin{equation}\label{prop-3}
\begin{aligned}
&\frac{d}{dx} B_{n}(x)=nB_{n-1}(x) \quad \text{and}\quad \frac{d}{dx} B_{n,\chi}(x)=nB_{n-1,\chi}(x),\quad n\geq1, \\
&\frac{d}{dx} \overline B_{n,\chi}(x)=n\overline B_{n-1,\chi}(x),\;n\geq2
\quad\text{and}\quad \overline B_{n,\chi}(f_\chi)=\overline B_{n,\chi}(0)=B_{n,\chi}, \\
&B_{n,\chi}(-x)=(-1)^n\chi(-1)B_{n,\chi}(x),\; n\geq0.
\end{aligned}
\end{equation}

The Euler polynomials $E_n(x)$ are defined by the generating function
\begin{equation}\label{def-E-pol}
\frac{2e^{tx}}{e^t+1}=\sum_{n=0}^\infty E_n(x)\frac{t^n}{n!}\quad(|t|<\pi)
\end{equation}
(see \cite{Sun}).
Suppose that $\chi$ is a primitive character modulo $f_\chi$ and  $f_\chi>1$ is odd. The generalized Euler polynomials $E_{n,\chi}(x)$
are defined by
\begin{equation}\label{def-E-poly}
\sum_{m=0}^{f_\chi-1}\frac{2(-1)^m\bar\chi(m)e^{(m+x)t}}{e^{f_\chi t}+1}=\sum_{n=0}^\infty E_{n,\chi}(x)\frac{t^n}{n!}\quad(|t|<\pi/f_\chi).
\end{equation}
 In particular, $E_{n,\chi}=E_{n,\chi}(0)$ are the generalized Euler numbers.
If $\chi^0$ is the principal character,
then $E_{n,\chi^0}(x)=E_{n}(x)$ for $n\geq0.$
For $0\leq x<1,$ $\overline E_{n,\chi}(x)$ denotes the character Euler function, with odd period $f_\chi,$ defined by Can and Da\u{g}l{\i} \cite{CanDa}.
Recall that (see \cite[(2.4) and (3.5)]{CanDa})
\begin{equation}\label{e-ft-rab}
\overline E_{n,\chi}(x)=f_\chi^{n}\sum_{m=0}^{f_\chi-1}(-1)^m\bar\chi(m)\overline E_n\left(\frac{m+x}{f_\chi}\right), \quad n\geq0
\end{equation}
and \begin{equation}\label{b-e-rel}
2^{n+1}\bar\chi(2)\overline B_{n+1,\chi}\left(\frac x2\right)-\overline B_{n+1,\chi}(x)=-\frac{n+1}2\overline E_{n,\chi}(x).
\end{equation}
\subsection{Main results} We now state our first main theorem which is an analogue of  \cite[Theorem 1.1]{Mao} and \cite[Theorem 1]{Mc}.

\begin{theorem}\label{thm-partial theta fts}
Let $G_{2,\chi}(\theta,b,r)$ be defined in (\ref{sums-theta-2}) and let $\chi$ be a nonprincipal character.
If $b\geq 0$, then for any nonnegative integer $N,$ we have the following asymptotic expansion
$$G_{2,\chi}(\theta,b,r)=
-r\sum_{n=0}^{N}(-1)^n\frac{B_{rn+1,\bar\chi}\left(\frac br\right)}{(rn+1)n!}
\theta^n+O(\theta^{N+1})$$  as $\theta\to0^+$.
\end{theorem}

We also have the following result.

\begin{theorem}\label{thm2}
Let $G_{1,\chi}(\theta,b,r)$ be defined in (\ref{sums-theta-1})
and let $\chi$ be a nonprincipal character. Suppose that the conductor $f_\chi$ of $\chi$ is odd.
If $0\leq b< r$, then for any nonnegative integer $N,$ we have the following asymptotic expansion
 $$G_{1,\chi}(\theta,b,r)=\frac r2\sum_{n=0}^{N}(-1)^{n}\frac{E_{rn,\bar\chi}\left(\frac br\right)}{n!}\theta^n+O(\theta^{N+1})$$
 as $\theta\to0^+.$
\end{theorem}

\begin{remark}
As pointed out by one of referees, in \cite{BR}, the authors considered a family of quantum modular forms of weight 3/2 coming from partial theta functions twisted with a Dirichlet character. Their companion is given by certain Eichler's integrals of weight 1/2 cusp forms. As we know the space of these forms is spanned by theta functions with a character (Serre--Stark's theorem).  But in order to establish quantum modularity, Bringmann and Rolen had to examine asymptotic expansion at other roots of unity (not only at zero). In summary, \cite{BR} presumably includes a version of Theorems \ref{thm-partial theta fts} and \ref{thm2} for ``weight" 3/2 partial theta functions twisted with a Dirichlet character. 
\end{remark}

\subsection{Some consequences} Before proving our main results, we show some of their consequences.

Let $\chi$ be a nonprincipal character.
Setting $r=1$ in (\ref{sums-theta-2}), from the summation formula for the geometric series we have
\begin{equation}\label{sums-theta-2-r=1}
\begin{aligned}
G_{2,\chi}(\theta,b,1)
&=\sum_{a=0}^{f_\chi-1}\sum_{m=0}^\infty\chi(mf_\chi+a)e^{-(mf_\chi+a+b)\theta} \\
&=\sum_{a=0}^{f_\chi-1}\chi(a)\sum_{m=0}^\infty e^{-(mf_\chi+a+b)\theta} \\
&=-\frac1{\theta}\sum_{a=0}^{f_\chi-1}\chi(a)e^{-(a+b)\theta}\frac{(-\theta)}{1-e^{-f_\chi\theta}}
\end{aligned}
\end{equation}
 for $|\theta|<2\pi/f_\chi$.
Then applying the generating function for the generalized Bernoulli polynomials (\ref{def-g-B-pol}), the asymptotic expansion of $G_{2,\chi}(\theta,b,1)$
at $\theta=0$ has the form
\begin{equation}\label{sym-gB}
G_{2,\chi}(\theta,b,1)= -\sum_{n=0}^\infty(-1)^n\frac{B_{n+1,\bar\chi}(b)}{(n+1)!}\theta^n,
\end{equation}
 for $|\theta|<2\pi/f_\chi$.

Let $r\in\mathbb{N}$. In the following, we remark that the asymptotic expansion in Theorem \ref{thm-partial theta fts}(ii)
is also established if we replace $b$  by $b+rf_{\chi}$. From  (\ref{sums-theta-2}) and Theorem \ref{thm-partial theta fts}(ii), we have
\begin{equation}\label{con-f-g2-add}
\begin{aligned}
G_{2,\chi}(\theta,b+rf_\chi,r) &=r\sum_{n=0}^\infty\chi(n+f_\chi)e^{-(n+b/r+f_\chi)^r\theta} \\
&=r\sum_{n=0}^\infty\chi(n)e^{-(n+b/r)^r\theta} -r\sum_{m=1}^{f_\chi}\chi(m)e^{-(m+b/r)^r\theta} \\
&= G_{2,\chi}(\theta,b,r)-r\sum_{m=1}^{f_\chi}\chi(m)e^{-(m+b/r)^r\theta} \\
&=-r\sum_{n=0}^{N}(-1)^n\frac{B_{rn+1,\bar\chi}\left(\frac br\right)}{rn+1}\frac{\theta^n}{n!}+O(\theta^{N+1})\\ 
&\quad-r\sum_{m=1}^{f_\chi}\chi(m)e^{-(m+b/r)^r\theta}.
\end{aligned}
\end{equation}
Then from the power series expansion of the function $e^{-(m+b/r)^r\theta}$,
we get
\begin{equation}\label{con-f-g2}
\begin{aligned}
G_{2,\chi}(\theta,b+rf_\chi,r) &=-r\sum_{n=0}^{N}(-1)^n\frac{B_{rn+1,\bar\chi}\left(\frac br\right)}{rn+1}\frac{\theta^n}{n!}+O(\theta^{N+1}) \\
&\quad - r\sum_{n=0}^{\infty}(-1)^n \sum_{m=1}^{f_\chi}\chi(m)\left(m+\frac br\right)^{rn}\frac{\theta^n}{n!}
 \\
&=-r\sum_{n=0}^{N}(-1)^n\biggl[\frac{1}{rn+1}B_{rn+1,\bar\chi}\left(\frac br\right) \\
&\quad\qquad\qquad\qquad +\sum_{m=1}^{f_\chi}\chi(m)\left(m+\frac br\right)^{rn}\biggl]\frac{\theta^n}{n!}+O(\theta^{N+1}).
\end{aligned}
\end{equation}

By the generating function of the generalized Bernoulli polynomials (\ref{def-g-B-pol}), we have
\begin{equation}\label{fox-4}
B_{n,\chi}(x+lf_\chi)-B_{n,\chi}(x)=n\sum_{m=1}^{lf_\chi}\bar\chi(m)(m+x)^{n-1},
\end{equation}
where $n\geq0$ and $l\in\mathbb N.$
Letting $l=1$ and $x=b/r,$ replacing $n$ by $rn+1,$ $\chi$ by $\bar\chi$ in (\ref{fox-4}), then substituting into (\ref{con-f-g2}),
we obtain the following identity
$$G_{2,\chi}(\theta,b+rf_\chi,r) =-r\sum_{n=0}^{N}(-1)^n
\frac{B_{rn+1,\bar\chi}\left(\frac br +f_\chi\right)}{rn+1}\frac{\theta^n}{n!} +O(\theta^{N+1}).$$

\section{Proofs}

Our main technical tool comes from Zagier's treatment of asymptotic expansions for infinite series, which can be found in Section 4 of \cite{Zag}.

First we need a character analogue of the Euler-Maclaurin summation formula due to Berndt \cite{Ber75}.

\begin{theorem}[{\cite[Theorem 4.1]{Ber75}}]\label{ch-E-M sum}
Let $f\in C^{(N+1)}[\alpha,\beta],-\infty<\alpha<\beta<\infty.$ Then
$$\begin{aligned}
\sideset{}{'}\sum_{\alpha\leq m\leq\beta}\chi(m)f(m)&=\chi(-1)\sum_{n=0}^N\frac{(-1)^{n+1}}{(n+1)!}\left(\overline B_{n+1,\bar\chi}(\beta)f^{(n)}(\beta)
-\overline B_{n+1,\bar\chi}(\alpha)f^{(n)}(\alpha) \right) \\
&\quad+\chi(-1)\frac{(-1)^N}{(N+1)!}\int_{\alpha}^{\beta}\overline B_{N+1,\bar\chi}(x)f^{(N+1)}(x) dx,
\end{aligned}$$
where the dash indicates that if $m=\alpha$ or $m=\beta,$ then only $\frac12\chi(\alpha)f(\alpha)$ or $\frac12\chi(\beta)f(\beta)$ is counted, respectively.
\end{theorem}

Let  $f:(0,\infty)\to\mathbb C$ be a smooth function which has an asymptotic power series expansion around 0. This means that
\begin{equation}\label{asy-exp}
f(t)=\sum_{n=0}^\infty b_nt^n
\end{equation}
as $t\to 0^+.$ Also, we assume that $f(t)$ and all of its derivatives rapidly decay at infinity, i.e., the function $t^Af^{(n)}(t)$ is
bounded on $\mathbb R_+$ for any $A\in\mathbb R$ and $n\in \mathbb Z^+$ (see \cite{Mao}).

For any $a\geq0,$ we consider the summation
\begin{equation}\label{def-ga}
g_{a,\chi}(t)=\ 
\sum_{m=0}^\infty \chi(m) f((m+a)t)
\end{equation}
and  in the next lemma we prove  that its asymptotic behaviour can be simply described in terms of the coefficients of the expansion (\ref{asy-exp}).

\begin{lemma}\label{lem1}
Suppose that $f$ has the asymptotic expansions (\ref{asy-exp}) and $f$ together with all of its derivatives are of rapid decay at infinity.
Suppose that $\chi$ is a nonprincipal character. If $a\geq 0,$ then for any nonnegative integer $N,$ the function $g_{a,\chi}(t)$ defined in (\ref{def-ga}) has the following asymptotic expansion
$$g_{a,\chi}(t)=-\sum_{n=0}^{N}b_n\frac{B_{n+1,\bar\chi}(a)}{n+1}t^n + O(t^{N+1}),$$ as $t\to0^+.$
\end{lemma}
\begin{proof}
Let $\alpha=0$ and $\beta=f_\chi M$ in Theorem \ref{ch-E-M sum}. Then
\begin{equation}\label{pf-E-M}
\begin{aligned}
\sum_{m=0}^{f_\chi M}\chi(m)f(m)&=\chi(-1)\sum_{n=0}^N\frac{(-1)^{n+1} B_{n+1,\bar\chi}}{(n+1)!}\left(f^{(n)}(f_\chi M)-f^{(n)}(0) \right) \\
&\quad+\chi(-1)\frac{(-1)^N}{(N+1)!}\int_{0}^{f_\chi M}\overline B_{N+1,\bar\chi}(x)f^{(N+1)}(x) dx,
\end{aligned}
\end{equation}
where we have used $\overline B_{n,\bar\chi}(f_\chi M)=\overline B_{n,\bar\chi}(0)=B_{n,\bar\chi}$ (see (\ref{prop-3}) above).
From the assumption,  the function $f$ and each of its derivatives are of rapid decay at infinity, we have $\int_0^\infty\left|f^{(N)}(x)\right|dx$ converges.
Since $\overline B_{N,\bar\chi}(x)$ is periodic and hence bounded, letting $M\to\infty$ in (\ref{pf-E-M}), we get
\begin{equation}\label{E-M-inf}
\begin{aligned}
\sum_{m=0}^{\infty}\chi(m)f(m)&=\chi(-1)\sum_{n=0}^N\frac{(-1)^{n} B_{n+1,\bar\chi}}{(n+1)!}f^{(n)}(0) \\
&\quad+\chi(-1)\frac{(-1)^N}{(N+1)!}\int_{0}^{\infty}\overline B_{N+1,\bar\chi}(x)f^{(N+1)}(x) dx.
\end{aligned}
\end{equation}
Replacing $f(x)$ by $f(tx)$ and then $x$ by $x/t$ with $t>0$ in the above equation, we have \begin{equation}\label{E-M-res}
\begin{aligned}
\sum_{m=0}^{\infty}\chi(m)f(mt)&=\chi(-1)\sum_{n=0}^N\frac{(-1)^{n} B_{n+1,\bar\chi}}{(n+1)!}f^{(n)}(0) t^n \\
&\quad+(-t)^N\chi(-1)\int_{0}^{\infty} \frac{\overline B_{N+1,\bar\chi}(x/t)}{(N+1)!}f^{(N+1)}(x) dx.
\end{aligned}
\end{equation}
From (\ref{asy-exp}), we have $f^{(n)}(0)=n!b_n$, substituting into (\ref{E-M-res}) we get 
\begin{equation}\label{E-M-res-add1}
\begin{aligned}
\sum_{m=0}^{\infty}\chi(m)f(mt)&=\chi(-1)\sum_{n=0}^{N}b_{n}\frac{(-1)^{n} B_{n+1,\bar\chi}}{n+1}t^n \\
&\quad+(-t)^N\chi(-1)\int_{0}^{\infty} \frac{\overline B_{N+1,\bar\chi}(x/t)}{(N+1)!}f^{(N+1)}(x) dx.
\end{aligned}
\end{equation}
From the same reason as the above, the last integral $\int_{0}^{\infty} \frac{\overline B_{N+1,\bar\chi}(x/t)}{(N+1)!}f^{(N+1)}(x) dx$ is bounded as $t\to0^+$ if $N$ is fixed,  so the last term in the above equation  can be denoted by $O(t^N).$
And by noticing that $$B_{n,\bar\chi}(0)=B_{n,\bar\chi}$$ and $$B_{n,\bar\chi}(-x)=(-1)^n\chi(-1)B_{n,\bar\chi}(x),~\textrm{for}~\; n\geq0$$  (see (\ref{prop-3}) above), we have
\begin{equation}
g_{0,\chi}(t)=-\sum_{n=0}^{N}b_n\frac{B_{n+1,\bar\chi}}{n+1}t^n + O(t^{N}).
\end{equation}
Since $N$ is arbitrary, we have
\begin{equation}
\begin{aligned}
g_{0,\chi}(t)&=-\sum_{n=0}^{N+1}b_n\frac{B_{n+1,\bar\chi}}{n+1}t^n + O(t^{N+1})\\
&=-\sum_{n=0}^{N}b_n\frac{B_{n+1,\bar\chi}}{n+1}t^n - b_{N+1}\frac{B_{N+2,\bar\chi}}{N+2}t^{N+1}+O(t^{N+1})\\
&=-\sum_{n=0}^{N}b_n\frac{B_{n+1,\bar\chi}}{n+1}t^n+O(t^{N+1}),
\end{aligned}
\end{equation}
which is the desired expansion in the case when $a=0$.

Now we show that  it is also established for $a>0$. If we write \begin{equation}\label{g-define} g(x)=f((x+a)t)~~\text{with}~~ a>0\end{equation}
and use (\ref{asy-exp}), then we get
\begin{equation}\label{g=f}
g(x)=f((x+a)t)=\sum_{n=0}^\infty b_n(x+a)^nt^n.
\end{equation}
Therefore,
\begin{equation}\label{g-diff}
g^{(j)}(x)=\sum_{n=j}^\infty b_n \binom nj j!(x+a)^{n-j}t^n.
\end{equation}
Putting $x=0$ in (\ref{g-diff}) and multipling both sides of the result equality by $\frac{B_{j+1,\bar\chi}}{(j+1)!}$, we have
\begin{equation}\label{b-diff}
\frac{B_{j+1,\bar\chi}g^{(j)}(0)}{(j+1)!}=\sum_{n=j}^N b_n\binom nj \frac{B_{j+1,\bar\chi}a^{n-j}t^n}{j+1} +O(t^{N+1}).
\end{equation}
Then multiplying  both sides of the above equality by $(-1)^{j}$ and adding from $j=0$ to $N$, we get
\begin{equation}\label{b-diff-rea}
\begin{aligned}
\sum_{j=0}^N\frac{(-1)^jB_{j+1,\bar\chi}g^{(j)}(0)}{(j+1)!}&=\sum_{j=0}^N(-1)^j\sum_{n=j}^N b_n\binom nj \frac{B_{j+1,\bar\chi}a^{n-j}t^n}{j+1} +O(t^{N+1}) \\
&=\sum_{n=0}^N b_n\left[\sum_{j=0}^n(-1)^j\binom nj \frac{B_{j+1,\bar\chi}a^{n-j}}{j+1} \right]t^n +O(t^{N+1}).
\end{aligned}
\end{equation}
Applying (\ref{E-M-inf}) with $f(x)$ replaced by $g(x),$ we have
\begin{equation}\label{E-M-inf-g}
\begin{aligned}
\sum_{m=0}^{\infty}\chi(m)g(m)&=\chi(-1)\sum_{j=0}^N\frac{(-1)^{j} B_{j+1,\bar\chi}}{(j+1)!}g^{(j)}(0) \\
&\quad+\chi(-1)\frac{(-1)^N}{(N+1)!}\int_{0}^{\infty}\overline B_{N+1,\bar\chi}(x)g^{(N+1)}(x) dx.
\end{aligned}
\end{equation}
By (\ref{g-define}),  (\ref{E-M-inf-g}) and (\ref{b-diff-rea}), we get
\begin{equation}\label{f=g-res}
\begin{aligned}
\sum_{m=0}^{\infty}\chi(m)f((m+a)t)
&=\sum_{m=0}^{\infty}\chi(m)g(m)\\
&=\chi(-1)\sum_{j=0}^N\frac{(-1)^{j} B_{j+1,\bar\chi}}{(j+1)!}g^{(j)}(0) \\
&\quad+\chi(-1)\frac{(-1)^N}{(N+1)!}\int_{0}^{\infty}\overline B_{N+1,\bar\chi}(x)g^{(N+1)}(x) dx\\
&=\chi(-1)\sum_{n=0}^N b_n\left[\sum_{j=0}^n(-1)^j\binom nj \frac{B_{j+1,\bar\chi}a^{n-j}}{j+1} \right]t^n  \\
&\quad+\chi(-1)\frac{(-1)^N}{(N+1)!}\int_{0}^{\infty}\overline B_{N+1,\bar\chi}(x)g^{(N+1)}(x) dx \\
&\quad +O(t^{N+1}).
\end{aligned}
\end{equation}
Since from  (\ref{g-define}) we have $g^{(N+1)}(x)=t^{(N+1)}f^{(N+1)}((x+a)t),$
the above equation becomes to
 \begin{equation}\label{f=g-res-add}
\begin{aligned}  
\sum_{m=0}^{\infty}\chi(m)f((m+a)t)
&=\chi(-1)\sum_{n=0}^N b_n\left[\sum_{j=0}^n(-1)^j\binom nj \frac{B_{j+1,\bar\chi}a^{n-j}}{j+1} \right]t^n  \\
&\quad+(-t)^N\chi(-1)\int_{at}^{\infty} \frac{\overline B_{N+1,\bar\chi}(x/t -a)}{(N+1)!}f^{(N+1)}(x) dx \\
&\quad +O(t^{N+1}).
\end{aligned}
\end{equation}
From the same reason as before, the integral $\int_{at}^{\infty} \frac{\overline B_{N+1,\bar\chi}(x/t -a)}{(N+1)!}f^{(N+1)}(x) dx$ is bounded as $t\to0^+$ with $N$ fixed, so the remainder term in the above expansion becomes $O(t^N).$
Using the identity
$$B_{n,\chi}(x)=\sum_{j=0}^n\binom nj B_{j,\chi}x^{n-j},$$
and $B_{0,\chi}=0$ if $\chi\neq\chi^0$ (see \cite[p.~9]{Iwasawa}), we have the expression
\begin{equation}\label{gB-exp}
\begin{aligned}
\sum_{j=0}^n(-1)^j\binom nj \frac{B_{j+1,\bar\chi}a^{n-j}}{j+1}
&=\frac1{n+1}\sum_{j=0}^n(-1)^j\binom{n+1}{j+1}{B_{j+1,\bar\chi}a^{n-j}} \\
&=\frac1{n+1}\sum_{j=-1}^n(-1)^j\binom{n+1}{j+1}{B_{j+1,\bar\chi}a^{n-j}}
 \\
&=\frac1{n+1}\sum_{j=0}^{n+1}(-1)^{j-1}\binom{n+1}{j}{B_{j,\bar\chi}a^{n-j+1}} \\
&=\frac{(-1)^n}{n+1}B_{n+1,\bar\chi}(-a).
\end{aligned}
\end{equation}
Finally, substituting  (\ref{gB-exp}) into (\ref{f=g-res-add}) and noticing that
$$B_{n,\bar\chi}(-x)=(-1)^n\bar\chi(-1)B_{n,\bar\chi}(x),\; n\geq0$$
(see (\ref{prop-3}) above), we obtain the following asymptotic formula $$g_{a,\chi}(t)=-\sum_{n=0}^{N}b_n\frac{B_{n+1,\bar\chi}(a)}{n+1}t^n + O(t^{N}).$$
Since $N$ is arbitrary,
we have  
\begin{equation}
\begin{aligned}
g_{a,\chi}(t)&=-\sum_{n=0}^{N+1}b_n\frac{B_{n+1,\bar\chi}(a)}{n+1}t^n + O(t^{N+1})\\
&=-\sum_{n=0}^{N}b_n\frac{B_{n+1,\bar\chi}(a)}{n+1}t^n 
-b_{N+1}
\frac{B_{N+2,\bar\chi}(a)}{N+2} t^{N+1}+O(t^{N+1})\\
&=-\sum_{n=0}^{N}b_n\frac{B_{n+1,\bar\chi}(a)}{n+1}t^n 
+O(t^{N+1})
\end{aligned}
\end{equation}
which completes our proof.
\end{proof}

\begin{proof}[Proof of Theorem \ref{thm-partial theta fts}.]
Let $r$ be any positive integer and let $\chi$ be a nonprincipal character.
Define \begin{equation}\label{f-define}
f(t)=e^{-t^r}~~\textrm{for}~~t>0.
\end{equation}
It is easy to see that this function is a smooth function and it is of rapid decay at infinity.
From the Taylor expansion
\begin{equation}\label{f-r-exp}
e^{-t^r}=\sum_{n=0}^\infty\frac{(-1)^nt^{rn}}{n!},
\end{equation}
we see that the function $f(t)=e^{-t^r}$ has the expansion form (\ref{asy-exp}) at $t=0$:
\begin{equation*}
f(t)=\sum_{m=0}^\infty b_mt^m
\end{equation*}
with
\begin{equation}\label{f-r-coe}
b_{m}=b_{rn}=\frac{(-1)^n}{n!}\text{ if }r\mid m,\text{ and } b_m=0 \text{ if }r\nmid m,
\end{equation}
for $m\geq0.$

Let $\theta=t^r$ with $r\in\mathbb N.$
From (\ref{sums-theta-2}), (\ref{f-define}), Lemma~\ref{lem1}(ii) and (\ref{f-r-coe}), for any positive integer $N,$ we have
\begin{equation}\label{eq b=r}
\begin{aligned}
G_{2,\chi}(\theta,b,r)&=r\sum_{m=0}^\infty\chi(m)e^{-((m+b/r)t)^r} \\
&=r\sum_{m=0}^\infty\chi(m)f((m+b/r)t)\\
&=-r\sum_{n=0}^{N}b_{rn}
\frac{B_{rn+1,\bar\chi}\left(\frac br\right)}{rn+1}
t^{rn}+O(t^{r(N+1)})\\
&(\text{the remainder term is}~ O(t^{r(N+1)})~\text{in the above}, \\
&\text{since by} ~(\ref{f-r-coe})~  \text{the coefficients of $t^{m}$ equals to 0 if $r\nmid m$)}\\
&=-r\sum_{n=0}^{N}(-1)^n\frac{B_{rn+1,\bar\chi}\left(\frac br\right)}{(rn+1)n!}
t^{rn}+O(t^{r(N+1)})\\
&=-r\sum_{n=0}^{N}(-1)^n\frac{B_{rn+1,\bar\chi}\left(\frac br\right)}{(rn+1)n!}
\theta^{n}+O(\theta^{N+1})
\end{aligned}
\end{equation}
as $\theta\to0^+.$
This completes the proof.
\end{proof}

\begin{proof}[Proof of Theorem \ref{thm2}.]
By separating even and odd terms, we find that
\begin{equation}\label{sums-theta-1-r=1}
\begin{aligned}
G_{1,\chi}(\theta,b,r)=\chi(2)G_{2,\chi}(2^r\theta,b/2,r)-r\sum_{n=0}^\infty\chi(2n+1)e^{-(2n+1+b/r)^r\theta}.
\end{aligned}
\end{equation}
By the definition of $G_{2,\chi}(\theta,b,r)$ (see (\ref{sums-theta-2}) above), we have
\begin{equation}\label{rel 1-2}
\begin{aligned}
G_{2,\chi}(\theta,b,r)-\chi(2)G_{2,\chi}(2\theta,b/2,r)=r\sum_{n=0}^\infty\chi(2n+1)e^{-(2n+1+b/r)^r\theta}.
\end{aligned}
\end{equation}
Substituting the above equality into (\ref{sums-theta-1-r=1}), we get
\begin{equation}\label{rel12-1}
\begin{aligned}
G_{1,\chi}(\theta,b,r)=2\chi(2)G_{2,\chi}(2^r\theta,b/2,r)-G_{2,\chi}(\theta,b,r).
\end{aligned}
\end{equation}
Letting $x=b/r$ with $0<b<r$ in (\ref{b-e-rel}), we find that
\begin{equation}\label{b-e-r-pf}
2^{n+1}\bar\chi(2)B_{n+1,\chi}\left(\frac{b/r}2\right)-B_{n+1,\chi}\left(\frac br\right)=-\frac{n+1}2E_{n,\chi}\left(\frac br\right),
\end{equation}
where we have used the fact that $\overline B_{n,\chi}(x)=B_{n,\chi}(x)$ and $\overline E_{n,\chi}(x)=E_{n,\chi}(x)$
for $0\leq x<1.$
Let $\chi$ be a  Dirichlet character with odd conductor $f_\chi.$
Substituting the result of Theorem \ref{thm-partial theta fts}(ii) into (\ref{rel12-1}), we have 
\begin{equation}\label{rel12-fin-add}
\begin{aligned}
G_{1,\chi}(\theta,b,r)&=r\sum_{n=0}^N\frac{(-1)^n}{rn+1}
\left[B_{rn+1,\bar\chi}\left(\frac br\right)-2^{rn+1}\chi(2)B_{rn+1,\bar\chi}\left(\frac{b/r}2\right) \right]\frac{\theta^n}{n!}\\
&\quad +O(\theta^{N+1}),
\end{aligned}
\end{equation} 
then applying (\ref{b-e-r-pf}), we get
\begin{equation}\label{rel12-fin}
G_{1,\chi}(\theta,b,r)=\frac r2\sum_{n=0}^{N}(-1)^{n}\frac{E_{rn,\bar\chi}\left(\frac br\right)}{n!}\theta^n+O(\theta^{N+1})
\end{equation}
as $\theta\to0^+.$ This completes our proof.
\end{proof}

\section{Connection with certain $L$-series}

Let $C:\mathbb{Z}\rightarrow\mathbb{C}$ be a periodic function with mean value 0 and $L(s,C)=\sum_{n=1}^{\infty}C(n)n^{-s}~~(\textrm{Re}(s)>1)$ be the associated $L$-series.  Lawrence and Zagie~\cite{LZ} proved that $L(s,C)$ can be extended
to the whole complex plane  $\mathbb{C}$, and by using Mellin transformation they also showed that  the two functions $\sum_{n=1}^{\infty} C(n)e^{-nt}$ and $\sum_{n=1}^{\infty} C(n)e^{-n^{2}t}~~(t>0)$ have the asymptotic expansions
$$
\sum_{n=1}^{\infty} C(n)e^{-nt}\sim\sum_{r=0}^{\infty}L(-r,C)\frac{(-t)^{r}}{r!}$$
and
$$\sum_{n=1}^{\infty} C(n)e^{-n^{2}t}\sim \sum_{r=0}^{\infty}L(-2r,C)\frac{(-t)^{r}}{r!}$$
as $t\to 0^+$. Furthermore, the number $L(-r,c)$
are given explicitly by  \begin{equation*}L(-r,C)=-\frac{M^{r}}{r+1}\sum_{n=1}^{M}C(n)B_{r+1}\left(\frac{n}{M}\right)\quad(r=0,1,\dots)\end{equation*}
where $B_{k}(x)$ denotes the $k$th Bernoulli polynomial and $M$ is any period of the function $C(n)$.

Define $$A(q)=\sum_{n\geq1}\chi_{+}(n)q^{(n^2-1)/120},$$ where $\chi_{+}(n)$ is periodic with period 60 and takes
the non-zero value 1 at 1, 11, 19, 29  and $-1$ at 31, 41, 49 and 59 and is 0 for all other $1\leq n\leq60.$
It is a holomorphic function in the unit disk. Let $\zeta$ be a root of unity. By using the above expansions, they showed that the radical limit of $1-\frac{1}{2}A(q)$ as $q$ tends to $\zeta$ equals $W(\zeta)$, the (rescaled) Witten-Reshetikhin-Turaev (WRT) invariant of the Poincar\'e
homology sphere. As pointed out by Lawrence and Zagier~\cite[p.~95]{LZ}, although for a general 3-manifold it is hopeless to give nice formulae for WRT invariants, the  WRT invariant of the Poincar\'e homology sphere is accessible to computations.

In this section, inspired by the works of Lawrence and Zagier~\cite{LZ}, we show  a connection between certain $L$-series and our asymptotic expansions.

Suppose $\epsilon\in\{1,2\}.$
Define the series $G_{\epsilon}(\theta,b,r)$ by
\begin{equation}\label{G-gen}
G_{\epsilon}(\theta,b,r)=r\sum_{n=0}^\infty(-1)^{\epsilon n}e^{-(n+b/r)^r\theta}.
\end{equation}
If $\chi$ is a nonprincipal character with conductor $f_\chi$, then define $G_{\epsilon,\chi}(\theta,b,r)$ by
\begin{equation}\label{G-gen-chi}
G_{\epsilon,\chi}(\theta,b,r)=r\sum_{n=0}^\infty(-1)^{\epsilon n}\chi(n)e^{-(n+b/r)^r\theta}.
\end{equation}
We remark here that $G_{\epsilon,\chi}(\theta,b,r)$ can be expressed in terms of the series $G_{\epsilon}(\theta,b,r).$
If $\chi$ is a character mod $f_\chi,$ then we rearrange the terms in the series for $G_{\epsilon,\chi}(\theta,b,r)$
according to the residue classes mod $f_\chi.$ That is, we write
$$n=mf_\chi+a,\quad\text{where }0\leq a\leq f_\chi-1\text{ and }m=0,1,2,\cdots,$$
and obtain
\begin{equation}\label{G-gen-chi-2}
\begin{aligned}
G_{\epsilon,\chi}(\theta,b,r)
&=r\sum_{a=0}^{f_\chi-1}\sum_{m=0}^\infty(-1)^{\epsilon(mf_\chi+a)}\chi(mf_\chi+a)e^{-(mf_\chi+a+b/r)^r\theta} \\
&=\sum_{a=0}^{f_\chi-1}(-1)^{\epsilon a}\chi(a)G_{\epsilon}\left(\theta f_\chi,\frac{ar+b}{f_\chi},r\right).
\end{aligned}
\end{equation}

Now we define the coefficients $\{\gamma_n\}$ from the asymptotic expansion
\begin{equation}\label{theta-L}
G_{\epsilon,\chi}(\theta,b,r)=r\sum_{n=0}^\infty(-1)^{\epsilon n}\chi(n)e^{-(n+b/r)^r\theta}\sim
\sum_{n=0}^\infty \gamma_n\theta^n \quad\text{as } \theta\to 0^+
\end{equation}
and consider the Mellin transform of $G_{\epsilon,\chi}(\theta,b,r)$:
\begin{equation}\label{theta-L-2}
\begin{aligned}
\int_0^\infty G_{\epsilon,\chi}(\theta,b,r)\theta^{s-1}d\theta&=r\sum_{n=0}^\infty(-1)^{\epsilon n}\chi(n)\int_0^\infty e^{-(n+b/r)^r\theta}\theta^{s-1}d\theta \\
&=r\sum_{n=0}^\infty\frac{(-1)^{\epsilon n}\chi(n)}{(n+b/r)^{rs}}\int_0^\infty e^{-t}t^{s-1}dt \\
&=\Gamma(s)L_{r,\epsilon}(rs,b;\chi).
\end{aligned}
\end{equation}
Here $L_{r,\epsilon}(s,b;\chi)$ is the modified Dirichlet $L$-series defined by
\begin{equation}\label{pm-l}
L_{r,\epsilon}(s,b;\chi)=r\sum_{m=0}^\infty\frac{(-1)^{\epsilon m}\chi(m)}{(m+b/r)^{s}},
\end{equation}
where $\chi$ is a primitive character modulo $f_\chi,b$ is a positive real number, and $r\in\mathbb N.$
On the other hand, we have
\begin{equation}\label{theta-L-3}
\begin{aligned}
\int_0^\infty G_{\epsilon,\chi}(\theta,b,r)\theta^{s-1}d\theta
&=\int_0^\infty\left(\sum_{n=0}^{N-1} \gamma_n\theta^n + O(\theta^N) \right)\theta^{s-1}d\theta \\
&=\sum_{n=0}^{N-1} \frac{\gamma_n}{s+n}+R_N(s),
\end{aligned}
\end{equation}
where $R_N(s)$ is an analytic function for Re$(s)>-N.$
And it is clear that the residue of  $\int_0^\infty G_{\epsilon,\chi}(\theta,b,r)\theta^{s-1}d\theta$ at $s=-n$ is $\gamma_n$ for $n < N$. Since $N$ is arbitrary, by (\ref{theta-L-2}), we have
\begin{equation}\label{res-L}
\gamma_n=\text{res}_{s=-n}\{L_{r,\epsilon}(rs,b;\chi) \Gamma(s) \}=\frac{(-1)^n}{n!}L_{r,\epsilon}(-rn,b;\chi)
\end{equation}
for $n\in\mathbb{N}$.
Therefore, from (\ref{theta-L}) and (\ref{res-L}) we conclude that
\begin{equation}\label{res-L-4}
G_{\epsilon,\chi}(\theta,b,r)\sim\sum_{n=0}^\infty\frac{(-1)^n}{n!} L_{r,\epsilon}(-rn,b;\chi)\theta^n\quad\text{as }\theta\to 0^+,
\end{equation}
where $\epsilon\in\{1,2\}.$
Then comparing the above equality with the results of Theorems \ref{thm-partial theta fts} and \ref{thm2}, we conclude that  $L_{r,\epsilon}(-rn,b;\chi)$ are given explicitly by
\begin{equation}\label{L-ex-vals}
L_{r,\epsilon}(-rn,b;\chi)=\begin{cases}
\frac12 rE_{rn,\bar\chi}\left(\frac br\right) &\text{if } \epsilon=1 \text{ and odd } f_\chi, \\
-r\frac{B_{rn+1,\bar\chi}\left(\frac br\right)}{rn+1} &\text{if } \epsilon=2,
\end{cases}
\end{equation}
where $\chi\neq\chi^0$ is a primitive character modulo $f_\chi,$ $b$ is a real number with $0<b<r$, $n$ is any nonnegative integer, and $r\in\mathbb N.$
Gathering  (\ref{res-L-4}) and (\ref{L-ex-vals}), we have established  the following theorem.

\begin{theorem}\label{res-L-asy}
Let the notations such as  $\epsilon$ be defined as above. Then
the series $G_{\epsilon,\chi}(\theta,b,r)$ have the asymptotic expansions
$$G_{\epsilon,\chi}(\theta,b,r)\sim\sum_{n=0}^\infty\frac{(-1)^n}{n!} L_{r,\epsilon}(-rn,b;\chi)\theta^n\quad\text{as }\theta\to 0^+.$$
Furthermore, the numbers $L_{r,\epsilon}(-rn,b;\chi)$ are given explicitly by
$$L_{r,\epsilon}(-rn,b;\chi)
=\begin{cases}
\frac12 rE_{rn,\bar\chi}\left(\frac br\right) &\text{if } \epsilon=1 \text{ and odd } f_\chi, \\
-r\frac{B_{rn+1,\bar\chi}\left(\frac br\right)}{rn+1} &\text{if } \epsilon=2,
\end{cases}$$
where $\chi\neq\chi^0$ is a primitive character modulo $f_\chi,$ $b$ is a real number with $0<b<r$, $n$ is any nonnegative integer, and $r\in\mathbb N.$
\end{theorem}
It is easy to see that
\begin{equation}\label{e-o-rel}
r\sum_{m=0}^\infty\frac{\chi(2m+1)}{(2m+1+b/r)^{s}}
=r\sum_{m=0}^\infty\frac{\chi(m)}{(m+b/r)^{s}}-r\sum_{m=0}^\infty\frac{\chi(2m)}{(2m+b/r)^{s}}.
\end{equation}
This yields, after simplification, (cf. (10.24) on p.~176 of \cite{On})
\begin{equation}\label{expli-L}
\begin{aligned}
L_{r,1}(s,b;\chi)&=r\sum_{m=0}^\infty\frac{(-1)^m\chi(m)}{(m+b/r)^{s}} \\
&=r\sum_{m=0}^\infty\frac{\chi(2m)}{(2m+b/r)^{s}}-r\sum_{m=0}^\infty\frac{\chi(2m+1)}{(2m+1+b/r)^{s}} \\
&=2^{-s+1}\chi(2)r\sum_{m=0}^\infty\frac{\chi(m)}{(m+b/2r)^{s}}
-r\sum_{m=0}^\infty\frac{\chi(m)}{(m+b/r)^{s}}.
\end{aligned}
\end{equation}
We have used (\ref{e-o-rel}) in the last step of the above equation. Then by (\ref{pm-l})  we obtain
\begin{equation}\label{expli-L2}
\begin{aligned}
L_{r,1}(s,b;\chi)&=2^{-s+1}\chi(2)L_{r,2}(s,b/2;\chi)-L_{r,2}(s,b;\chi).
\end{aligned}
\end{equation}
Finally, substituting  $s=-rn$ into (\ref{expli-L2}), and from Theorem \ref{res-L-asy} we have the following result
which is equivalent  to \cite[p. 1209, (3.5)]{CanDa} and (\ref{b-e-rel}).
\begin{corollary}\label{B-K}
Let $\chi,$ etc., be as above. Then
$$E_{rn,\bar\chi}\left(\frac br\right)=\frac1{rn+1}\left(B_{rn+1,\bar\chi}\left(\frac br\right)-2^{rn+1}\chi(2)B_{rn+1,\bar\chi}\left(\frac b{2r}\right)\right),$$
where $\chi\neq\chi^0$ is a primitive character modulo $f_\chi,$ $b$ is a real number with $0<b<r$, $n$ is any nonnegative integer, and $r\in\mathbb N.$
\end{corollary}

\section*{Acknowledgement} The authors are enormously grateful to the anonymous referees whose comments and suggestions lead to a large improvement of the paper.

Min-Soo Kim is supported by the National Research Foundation of Korea(NRF) grant funded by the Korea government(MSIT) (No. 2019R1F1A1062499). Su Hu is supported by the Natural Science Foundation of Guangdong Province, China (No. 2020A1515010170).

\bibliography{central}

\end{document}